\documentclass[12pt]{article}
\usepackage[utf8]{inputenc}
\usepackage[T1]{fontenc}
\usepackage{amsmath, amsthm,enumerate}
\usepackage{amssymb,graphicx}
\usepackage{fullpage,mathtools}
\usepackage{xcolor}
\usepackage{cancel}
\usepackage[numbers,sort&compress]{natbib}

\usepackage[pdftex,breaklinks,pdfpagemode={None},pdfstartview={FitH},
            pdfview={FitH},colorlinks,linkcolor={black},
            citecolor={black}, urlcolor={black}]{hyperref}

\usepackage{tikz}
\usetikzlibrary{graphs,decorations.pathmorphing, decorations.pathreplacing, angles, quotes, shapes, arrows, topaths, calc, patterns, backgrounds, positioning, arrows, shapes.multipart}

\tikzset{st/.style = {circle,draw = blue!40,fill = blue!40,
  inner sep = 1pt, outer sep = 0pt}}
\tikzset{stt/.style = {circle,draw = teal!100,fill = teal!100,
  inner sep = 1.5pt, outer sep = 0pt}}
\tikzset{stt1/.style = {circle,draw=black,
  inner sep = 1.5pt, outer sep = 0pt}}    
\tikzset{inn/.style = {circle,draw = black!20,fill = black!10,
  inner sep = 2pt, outer sep = 0pt}}
\tikzset{nwtext/.style = {draw = white,fill = white,
    anchor= north west,inner sep = 2pt, outer sep = 0pt}}
\tikzset{wtext/.style = {draw = white,fill = white,
    anchor= west,inner sep = 2pt, outer sep = 0pt}}
\tikzset{elip/.style={ellipse,draw=black,fill=blue!15,
   inner sep = 2pt, outer sep = 2pt}}
\tikzset{min/.style={inner sep = 0pt, outer sep = 0pt}}
\newcommand{\email}[1]{\href{mailto:#1}{\tt#1}}

\newtheorem{theorem}{Theorem}[section]

\newtheorem{lemma}[theorem]{Lemma}

\usepackage[mathlines]{lineno}
\usepackage{etoolbox} %

\newcommand*\linenomathpatch[1]{%
  \cspreto{#1}{\linenomath}%
  \cspreto{#1*}{\linenomath}%
  \csappto{end#1}{\endlinenomath}%
  \csappto{end#1*}{\endlinenomath}%
}
\newcommand*\linenomathpatchAMS[1]{%
  \cspreto{#1}{\linenomathAMS}%
  \cspreto{#1*}{\linenomathAMS}%
  \csappto{end#1}{\endlinenomath}%
  \csappto{end#1*}{\endlinenomath}%
}

\expandafter\ifx\linenomath\linenomathWithnumbers
  \let\linenomathAMS\linenomathWithnumbers
  \patchcmd\linenomathAMS{\advance\postdisplaypenalty\linenopenalty}{}{}{}
\else
  \let\linenomathAMS\linenomathNonumbers
\fi

\linenomathpatch{equation}
\linenomathpatchAMS{gather}
\linenomathpatchAMS{multline}
\linenomathpatchAMS{align}
\linenomathpatchAMS{alignat}
\linenomathpatchAMS{flalign}

\newcommand{\sq}[1]{\overline{#1^2}}
\newcommand{\dist}{\textup{dist}}

\begin{document}
\title{There are no nontrivial chordal square-complementary graphs}
\author{Pierre Aboulker$^{1}$ \and Martin Milani{\v c}$^{2,5}$ \and
  Peter Muršič$^{3,6}$ \and Miguel A. Pizaña$^{4,7}$}

\maketitle
\footnotetext[1]{DIENS, \'Ecole normale sup\'erieure, CNRS, PSL University, Paris, France.  
E-mail: \email{pierreaboulker@gmail.com}.}
\footnotetext[2]{FAMNIT and IAM, University of Primorska. 
E-mail: \email{martin.milanic@upr.si}.}
\footnotetext[3]{FAMNIT, University of Primorska.  
E-mail: \email{peter.mursic@upr.si}.}
\footnotetext[4]{Universidad Autónoma Metropolitana. 
E-mail: \email{mpizana@gmail.com}.}
\footnotetext[5]{Partially supported by the Slovenian Research and Innovation Agency (I0-0035, research program P1-0285 and research projects J1-60012, J1-70035, J1-70046, and N1-0370).}
\footnotetext[6]{Partially supported by the Slovenian Research and Innovation Agency (research project N1-0370).}
\footnotetext[7]{Partially supported by CONAHCYT, grant A1-S-45528.}

\renewcommand{\thefootnote}{\arabic{footnote}}

\begin{abstract}
 We study \emph{square-complementary} graphs $G$ (satisfying $G^2 \cong \overline{G}$).  
  We show that in such graphs, no two vertices have comparable closed neighborhoods.  
  This implies that nontrivial square-complementary graphs have no simplicial vertices and are not chordal, thus solving two open problems posed in Discrete Mathematics \textbf{327} (2014) \hbox{62--75}.  
  We also show that no nontrivial square-complementary graph is distance-hereditary.
\end{abstract}

\maketitle

\section{Introduction}

A \emph{square-complementary} graph (\emph{squco} for short) is a graph whose square and complement are isomorphic.
Square-complementary graphs were introduced independently by  Schuster~\cite{Sch81} and by Akiyama, Era, and Exoo~\cite{AEE81} and further studied by several other
authors~\cite{CKR82,CK95,CLR89,DMP21,JP22,MPPV14}.

Milani\v{c} et al.~posed several open problems on squco graphs~\cite{MPPV14}, including whether a nontrivial squco graph could be chordal (Problem~5) or contain simplicial vertices (Problem~6).
Here we answer both questions in the negative.  
Moreover, we also prove that a nontrivial squco graph cannot be distance-hereditary.  
All these results generalize the fact that there are no nontrivial squco trees; see~\cite{BST94}.
Furthermore, the result for chordal graphs generalizes the fact that there are no nontrivial squco split graphs, interval graphs, or block graphs (see~\cite{MPPV14}), while the result for distance-hereditary graphs generalizes the fact that there are no nontrivial squco cographs (see~\cite{MPPV14}).

The fact that there are no nontrivial chordal squco graphs was mentioned in~\cite[p.~185]{Pri95}, referring to a preprint of Capobianco and Kim bearing the same title as the paper~\cite{CK95}.
However, this result is not mentioned in~\cite{CK95} and we are not aware of any published proof of this result.

\section{Preliminaries}

We consider only simple, finite, and undirected graphs.  
We use standard terminology and notation for the \emph{vertex set $V(G)$}, the \emph{edge set $E(G)$}, the \emph{open neighborhood $N(x)=N_{G}(x)$}, the \emph{closed neighborhood $N[x]=N_{G}[x]$} and the \emph{distance $\dist(x,y)=\dist_{G}(x,y)$}.  
Also, we define the \emph{ball of radius two} at a vertex $x$, by $B_2(x)=B_{2}^{G}(x) =\{y \in V(G)\colon \dist(x,y)\leq 2\}$. 
A graph is \emph{nontrivial} if it has at least two vertices, and \emph{trivial}, otherwise.

A vertex $x$ is \emph{simplicial} if $N[x]$ induces a complete
subgraph of $G$; it is \emph{pendant} if $|N(x)|=1$.  
A vertex \emph{$x$ is dominated by} a vertex $y$ if $N[x]\subseteq N[y]$; when $x\neq y$, we say that \emph{$x$ is dominated}.  
Two vertices $x,y$ are said to be \emph{true twins} if $N[x]=N[y]$; they are said to be \emph{false twins} if $N(x)=N(y)$.  
We say that $x$ and $y$ are \emph{twins} in any of these two cases. 
Any vertex is a twin of itself; however, we say that a vertex $x$ is a \emph{twin} only when there is a different vertex $y$ such that $x$ and $y$ are twins. 

Given a vertex $x$ in a connected graph $G$, the \emph{eccentricity} of $x$ in $G$ is the maximum distance from $x$ to any other vertex in $G$.  
The minimum eccentricity of a vertex of $G$ is its \emph{radius}, denoted $r(G)$, and the maximum eccentricity is its \emph{diameter}, denoted $d(G)$. 
Given a set $X\subseteq V(G)$, we denote by $G-X$ the subgraph of $G$ induced by $V(G)\setminus X$.

The \emph{square} of a graph $G$ is the graph denoted by $G^2$ and obtained from $G$ by adding all edges connecting pairs of vertices at distance~$2$.  
The \emph{complement} $\overline{G}$ of $G$ is the graph with vertex set $V(G)$, in which two distinct vertices are adjacent if and only if they are non-adjacent in $G$.  
We say that a graph $G$ is \emph{square-complementary} (\emph{squco} for short) if $G^2\cong \overline{G}$, where $\cong$ denotes the graph isomorphism relation.  
Note that a graph $G$ is squco if and only if $G$ is isomorphic to $\sq{G}$. 
Also, in a squco graph $G$, $xy\in E(\sq{G})$ if and only if $\dist_{G}(x,y)\geq 3$; hence, $N_{\sq{G}}(x) = \overline{B_{2}^{G}(x)}$ and it follows that $B_{2}^{G}(x)\subseteq B_{2}^{G}(y)$ if and only if $N_{\sq{G}}(y)\subseteq N_{\sq{G}}(x)$.  
Squco graphs are known to be connected and free of pendant vertices; nontrivial squco graphs are known to have radius 3~\cite{BST94}. 
More precisely, the following theorem is a consequence of~\cite[Lemmas 1.1 and 1.2]{BST94}.

\begin{theorem}\label{thm:squco-radius-diameter}
  If $G$ is a nontrivial squco graph, then $G$ is a connected graph with $r(G) = 3$ and $d(G)\in \{3,4\}$.
\end{theorem}

\emph{Chordal graphs} are graphs where every cycle of length at least~$4$ has a chord.  
Chordal graphs are known to contain simplicial vertices~\cite{Dir61}.  
\emph{Distance-hereditary} graphs are graphs in which the distances in every connected induced subgraph are the same as they are in the original graph. 
Nontrivial distance-hereditary graphs are known to contain either pendant vertices or twins (either false or true) \cite{BM86}.

Let $W$ be a finite set. 
A \emph{permutation} of $W$ is a bijective function $\varphi:W\rightarrow W$. 
The \emph{identity permutation} of $W$ is denoted by $1_{W}$. 
The \emph{product of two permutations} is simply their composition. 
The minimum positive integer $r$ such that $\varphi^{r} =1_{W}$ is called the \emph{exponent} of $\varphi$. 
Since $W$ is finite, every permutation of $W$ has an exponent. 
Given $x\in W$ and $X\subseteq W$, the \emph{$\varphi$-orbit of $x$} is ${\varphi\cdot x =\{\varphi^{n}(x)\colon n\in \mathbb{Z}\}}$ and the \emph{$\varphi$-orbit of $X$} is ${\varphi\cdot X =\{\varphi^{n}(X)\colon n\in \mathbb{Z}\}}$.
Since $W$ is finite, given $x$ and $X$ as before, there are $s,t\in \mathbb{Z}^{+}$ such that $\varphi^{s}(x)=x$ and $\varphi^{t}(X)=X$, hence, $\varphi\cdot x =\{x, \varphi(x), \varphi^{2}(x), \ldots,\varphi^{s-1}(x)\}$ and $\varphi\cdot X =\{X, \varphi(X), \varphi^{2}(X), \ldots,\varphi^{t-1}(X)\}$. 
Note that the exponent, $r$, of $\varphi$ satisfies the previous conditions (in the place of $s$ or $t$), but we usually take $s$ and $t$ to be the minimum positive integers satisfying the previous conditions. 
Finally, a set $X\subseteq W$ is said to be \emph{$\varphi$-invariant} if $\varphi(X) =X$.

\section{Results}

We say that a pair of different vertices, $\{x,y\}$, is a \emph{true dominating pair} if $N[x]\subseteq N[y]$ and that it is a \emph{false dominating pair} if $N(x) \subseteq N(y)$.  
Note that in true dominating pairs, $x$ and $y$ are necessarily adjacent, while in false dominating pairs, $x$ and $y$ are necessarily non-adjacent.
Dominated vertices belong to at least one dominating pair.  Also, note that the pair is a set and hence it does not know which of its vertices are dominated.

\begin{lemma}\label{lem:notruedompairs}
  Let $G$ be a squco graph and $\{x,y\}$ a dominating pair (either true or false) in~$G$.  
  Then, $\{x,y\}$ is a false dominating pair in $\sq{G}$ and $G$ does not contain true dominating pairs.
\end{lemma}

\begin{proof}
  Note that $G$ is nontrivial.  
  Being a squco graph, it follows from Theorem~\ref{thm:squco-radius-diameter} that $G$ is connected and has radius $3$.  
  If $\{x,y\}$ is a dominating pair, it follows (without loss of generality) that $B_{2}(x)\subseteq B_{2}(y)$, but
  then we have that $N(y)\subseteq N(x)$ in $\sq{G}$.  
  It follows that $\{x,y\}$ is a false dominating pair in $\sq{G}$.  
  If $G$ contained true dominating pairs, then $\sq{G}$ would contain more false dominating pairs than $G$ does, contrary to the hypothesis that $G$ and $\sq{G}$ are isomorphic.
\end{proof}

As immediate consequences we get:

\begin{theorem}\label{thm:notruetwins}
  If $G$ is nontrivial and squco then $G$ does not contain dominated vertices, true twins, nor simplicial vertices. \qed
\end{theorem}

\begin{theorem}\label{thm:nochordal}
  If $G$ is nontrivial and squco, then $G$ is not chordal. \qed
\end{theorem}

Note: false twin vertices in $G$ are also false twin vertices in $\sq{G}$.  
Since the two graphs are isomorphic, the sets of all twin vertices of $G$ and $\sq{G}$ coincide.

We say that a graph $G$ is a \emph{minimal nontrivial squco graph} if it is a nontrivial squco graph that does not contain any proper induced nontrivial squco subgraph.  
Such graphs cannot have twins of any kind:

\begin{theorem}\label{thm:notwins}
  Let $G$ be a minimal nontrivial squco graph. 
  Then $G$ contains no (true or false) twin vertices. 
  In particular, there is no nontrivial distance-hereditary squco graph.
\end{theorem}

\begin{proof}
  Let $\varphi\colon \sq{G}\rightarrow G$ be an isomorphism viewed as a permutation of $V(G)$. 
  By Theorem~\ref{thm:notruetwins}, any pair of twins in $G$ must be false twins.  
  Suppose $G$ contains false twins.  
  Let $x\in V(G)$ be a twin, and let $T_1$ be the set of all false twins of $x$ (including $x$).  
  Let $\{T_{1}, T_{2}, \ldots, T_{s}\}$ be the $\varphi$-orbit of $T_{1}$. 
  We may assume that $T_{i} = \varphi^{i-1}(T_{1})$ for
  $i=1,2,\ldots, s$, and that $\varphi^{s}(T_{1}) = T_{1}$, with $s$ being the minimum such positive integer.  
  Since $\varphi$ is an isomorphism, each of these sets $T_{i}$ is a maximal set of pairwise twin vertices.
  Furthermore, since the relation of being false twins is an
  equivalence relation, these sets are pairwise disjoint. 
  Moreover, $|T_{i}|\geq 2$ for all $i$.

  Let $y=\varphi^s(x)$. 
  Suppose first that $y\neq x$.  
  Since $\varphi^{s}(T_{1}) = T_{1}$, we have $y\in T_{1}$, so $y$ is a twin of $x$.  
  Hence, the transposition $\alpha = (x\,\,\, y)$ is an automorphism of $G$.  
  It follows that $\psi = \alpha\circ \varphi \colon\sq{G} \rightarrow G$ is also an isomorphism. 
  Since $\alpha$ only exchanges $x$ and $y$, we have that $\psi(z)$ is identical to $\varphi(z)$ except when $\varphi(z)\in \{x,y\}$. 
  It follows that $\psi^{i}(x) = \varphi^{i}(x)$ for $i=0,1,\ldots, s-1$ and that $\psi^{s}(x) = \psi(\psi^{s-1}(x)) = \psi(\varphi^{s-1}(x)) = \alpha(\varphi(\varphi^{{s-1}}(x))) = \alpha(\varphi^{s}(x)) = \alpha(y) = x$. 
  Hence, the $\psi$-orbit of $x$ is exactly $\{x, \psi(x), \psi^2(x), \ldots, \psi^{s-1}(x)\}$.  
  Therefore, we may assume without loss of generality that $y=x$ and that the $\varphi$-orbit of $x$ is exactly
  $X \coloneqq\{x, \varphi(x), \varphi^2(x), \ldots, \varphi^{s-1}(x)\}$.

  Now, since each vertex in $X$ has a twin outside $X$, it follows that $G-X$ is connected and that the distances in $G-X$ are the same as in $G$ for every pair of vertices $x,y\in V(G) \setminus X$. 
  Let us show that the restriction of $\varphi$ to $V(G)\setminus X$ is an isomorphism from $\sq{(G-X)}$ to $G-X$. 
  Note first that $V\left(\sq{(G-X)}\right)=V(G) \setminus X =V(G-X)$. Since $X$ is $\varphi$-invariant, so is $V(G)-X$ and it follows that  $\varphi\left(V\left(\sq{(G-X)}\right)\right)=V(G-X)$. 
  Now take $x,y\in V(G) \setminus X$. 
  Then $xy\in E\left(\sq{(G-X)}\right)$ $\iff$
    $d_{G-X}(x,y)\geq 3$ $\iff$ $d_{G}(x,y)\geq 3$ $\iff$
        $xy\in E(\sq{G})$ $\iff$ $\varphi(x)\varphi(y)\in E(G)$ $\iff$
            $\varphi(x)\varphi(y)\in E(G-X)$.  
        Hence, $\sq{(G-X)}\cong G-X$, that is, $G-X$ is squco.  
        By the minimality of $G$, the graph $G-X$ must be trivial.  
        Recall that $T_{1}, T_{2},\ldots, T_{s}$ are pairwise disjoint, and that for all $i$, $|T_{i}|\geq 2$, but $|T_{i}\cap X| = 1$.  
            It follows that $s=1$ and hence, $G$            must have exactly two vertices, but there is no squco graph on two vertices.

  Finally, suppose that there exists a nontrivial distance-hereditary squco graph~$G_0$. Then, $G_0$ contains, as an induced subgraph, some minimal nontrivial squco graph $G$. 
Since the class of distance-hereditary graphs is closed under vertex deletions, $G$ is distance-hereditary. 
  Being distance-hereditary, it is well known that $G$
  contains either pendant vertices (which is incompatible with squco graphs; see~\cite{BST94}) or twins. 
  If $G$ contains twins, removing an orbit of twins as before would produce a smaller distance-hereditary graph $G-X$ which is also a squco graph. 
  As before, it follows that $G-X$ is trivial and that $G$ has exactly two vertices, a contradiction.
\end{proof}

\section{Alternate proof of Theorem~\ref{thm:nochordal}}

We found an alternate proof for the fact that there are no nontrivial chordal squco graphs, which we think provides additional insight on the issue. 
For this, we need some additional terminology and results from the literature.

Given a connected graph $G$, a set $S\subseteq V(G)$, and a vertex $w\in V(G)\setminus S$, the \emph{distance} in $G$ from $S$ to $w$ is the length of a shortest path in $G$ from a vertex in $S$ to $w$. 
Given two disjoint sets $X,Y\subseteq V(G)$, we say that $X$ and $Y$ are \emph{complete to each other} if every vertex in $X$ is adjacent to every vertex in $Y$, and \emph{anticomplete to each other} if no vertex in $X$ is adjacent to any vertex in $Y$.

\begin{theorem}[{\cite[Theorem~3.5]{CN84}}]\label{thm:chordal-radius-diameter}
If $G$ is a connected chordal graph, then $d(G)\ge 2r(G)-2$.
\end{theorem}

Given a graph $G$ and two non-adjacent vertices $x,y\in V(G)$, a \emph{minimal $x,y$-separator} is an inclusion-minimal set $S\subseteq V(G)$ such that $x$ and $y$ are in different connected components of $G-S$. 
A \emph{minimal separator} in a graph $G$ is a set $S\subseteq V(G)$ that is a minimal $x,y$-separator for some
non-adjacent vertex pair $x,y$.  
Given a graph $G$ and a set $X\subseteq V(G)$, we denote by $N_G(X)$ the set of all vertices in $V(G)\setminus X$ having a neighbor in $X$. 
The following characterization of minimal separators in graphs is well known (and easy to prove).

\begin{lemma}\label{lem:minimal-separators}
  Let $G$ be a graph. 
  A set $S\subseteq V(G)$ is a minimal separator in $G$ if and only if there exists two distinct components $C$ and
  $D$ of $G-S$ such that \hbox{$N_G(V(C)) = N_G(V(D)) = S$}.
\end{lemma}

Minimal separators in chordal graphs are characterized as follows.

\begin{theorem}[\cite{Dir61}]\label{thm:Dirac} Every minimal separator in a chordal graph is a clique.
\end{theorem}

Now we are ready to present:\bigskip

\textbf{Alternate proof of Theorem~\ref{thm:nochordal}}.  Suppose for a contradiction that $G$ is a chordal squco graph with $G\ncong K_1$.
By Theorem~\ref{thm:squco-radius-diameter}, $G$ is a connected graph with radius $3$ and diameter either $3$ or $4$.  
Since $r(G) = 3$ and $G$ is chordal, Theorem~\ref{thm:chordal-radius-diameter} implies that $d(G)\ge 4$.  
Therefore, $d(G) = 4$.  
Let $u$ and $v$ be a pair of vertices in $G$ at distance $4$ and let $(u = v_0,v_1,v_2,v_3,v_4 = v)$ be a shortest $u,v$-path in $G$.  
For all $i\in \{0,1,\ldots, 4\}$, let $L_i$ be the set of vertices of $G$ at distance exactly $i$ from $u$.  Then, $\{L_0, L_1, L_2, L_3, L_4\}$ is a partition of $V(G)$ into pairwise disjoint non-empty sets (note that $v_i\in L_i$ for all $i$).

Let $G'$ be the subgraph of $G$ induced by $L_3\cup L_4$ and let $C$ be the component of $G'$ containing $v$.  
Moreover, let $S = N_G(V(C))$.  
We claim that $S\subseteq L_2$.  
Indeed, the definition of $C$ implies that no vertex in $C$ has a neighbor in $(L_3\cup L_4)\setminus V(C)$.  
Since $S\subseteq V(G)\setminus V(C)$, this implies that no vertex of $S$ belongs to $L_3\cup L_4$.
Furthermore, since the sets $L_i$ are the distance layers from $u$, no vertex in $C$ has a neighbor in $L_0\cup L_1$; this implies that no vertex of $S$ belongs to $L_0\cup L_1$.  
Hence, $S\subseteq L_2$, as claimed.

Note that $C$ is the component of $G-S$ containing $v$.  
Let $D$ be the component of $G-S$ containing $u$.  
The definition of $S$ implies that $C\neq D$. 
We claim that $N_G(V(D)) = S$.  
The definition of $D$ implies that $N_G(V(D)) \subseteq S$.  
Moreover, if $w\in S$, then $w$ is in $L_2$ and is therefore adjacent to some vertex $z$ in $L_1$.
Since $z$ is adjacent to $u$ and $z\not\in S$, we have $z\in V(D)$.
Thus, every vertex in $S$ has a neighbor in $V(D)$, which implies that $S\subseteq N_G(V(D))$, thus establishing the claimed equality $N_G(V(D)) = S$.

Since $C$ and $D$ are distinct components of the graph $G-S$ with $N_G(V(C)) = N_G(V(D)) = S$, Lemma~\ref{lem:minimal-separators} implies that $S$ is a minimal separator in $G$.  In turn, Theorem~\ref{thm:Dirac} implies that $S$ is a clique.

To complete the proof, it will be useful to consider a different partition of $V(G)$ into five parts.  
For $i\ge 1$, let $A_i$ denote the set of vertices in $C$ that are at distance $i$ from $S$ in $G$ and let $B_i$ denote the set of vertices in $V(G)\setminus (V(C)\cup S)$ that are at distance $i$ from $S$ in $G$.
Then $A_1 = V(C)\cap L_3$ and $A_2 = V(C)\cap L_4$.  
Since $v_3\in A_1$ and $v= v_4\in A_2$, sets $A_1$ and $A_2$ are non-empty.
Analogously, since $v_2\in S$ and $v_1\in L_1$ is adjacent to $v_2$, we have $v_1\in B_1$ and, similarly, $u = v_0\in B_2$. 
Hence, sets $B_1$ and $B_2$ are also non-empty.

Next, we claim that $A_i = \emptyset$ for all $i\ge 3$.  Indeed, if this were not the case, then there would exist a vertex in $A_3$, but such a vertex would be at distance $5$ from $u$, contradicting the fact that $G$ has diameter $4$.  
Similarly, $B_i = \emptyset$ for all $i\ge 3$.  
It follows that $\{A_2, A_1, S, B_1, B_2\}$ is a partition of $V(G)$ into pairwise disjoint non-empty sets.

Let us now analyze the adjacencies in $\sq{G}$ between
vertices belonging to different sets in this partition.  
By construction, for every vertex pair $x,y\in V(G)$ such that $x\in A_2$ and $y\in B_1\cup B_2$, the distance in $G$ from $x$ to $y$ is at least $3$.
This implies that any such pair forms an edge in $\sq{G}$.  
Similarly, if $x\in A_1$ and $y\in B_2$, then $d_G(x,y) \ge 3$ and hence $x$ and $y$ are adjacent in $\sq{G}$.  
On the other hand, if $x\in A_1$ and $y\in S$, then the facts that $x$ is adjacent in $G$ to a vertex in $S$ and $S$ is a
clique in $G$ imply that the distance in $G$ between $x$ and $y$ is at most $2$.  
Hence $A_1$ and $S$ are anticomplete to each other in $\sq{G}$ and, similarly, so are $B_1$ and $S$.

Next, we show that one of $A_2$ and $B_2$ is a clique in $\sq{G}$. 
Suppose that $A_2$ contains two distinct vertices, say $a,a'$, that are non-adjacent in $\sq{G}$, and that $B_2$ contains two distinct vertices, say $b,b'$, that are non-adjacent in $\sq{G}$.  
Then the vertex set $\{a,a',b,b'\}$ induces a $C_4$ in $\sq{G}$.
It follows that $\sq{G}$ is not chordal, contradicting the fact that $\sq{G}$ is isomorphic to the chordal graph $G$.

Note that since $\sq{G}\cong G$, the graph $\sq{G}$ is connected.  
We complete the proof by showing that $r(\sq{G}) \le 2$.  
This will suffice, as it will imply that $r(\sq{G}) \neq r(G)$ and hence the two graphs cannot be isomorphic.

Suppose first that $A_2$ is a clique in $\sq{G}$.  
We claim that in this case, vertex $v$ has eccentricity at most $2$ in $\sq{G}$ (recall that $v\in A_2$).  
Since $A_2$ is a clique in $\sq{G}$ and $A_2$ is complete in $\sq{G}$ to $B_1\cup B_2$, the neighborhood of $v$ in $\sq{G}$ contains all vertices in $(A_2\setminus \{v\})\cup B_1\cup B_2$.  
Since $A_1$ is complete to $B_2$ in $\sq{G}$, every vertex in $A_1$ is at distance at most $2$ from $v$ in $\sq{G}$.  
Finally, consider a vertex $x\in S$.  
Since $S$ is a clique in $G$, it is an independent set in $\sq{G}$.
Therefore, in the graph $\sq{G}$, vertex $x$ cannot be adjacent to any vertex in $S$.  
Moreover, since $S$ is anticomplete in $\sq{G}$ to $A_1\cup B_1$, we infer that all the neighbors of $x$ in $\sq{G}$ are in
$A_2\cup B_2$.  
Since $\sq{G}$ is connected, it contains a vertex $y\in A_2\cup B_2$ adjacent to $x$, which implies that $x$ is
at distance at most $2$ from $v$ in $\sq{G}$.  
We showed that every vertex in $\sq{G}$ is at distance at most $2$ from $v$, that is, $v$ has eccentricity at most $2$ in $\sq{G}$, as claimed.  
The case when $B_2$ is a clique in $\sq{G}$ is similar.  
This completes the alternate proof of Theorem~\ref{thm:nochordal}. \qed

\bibliographystyle{mapbib4} 
\bibliography{squco} 

\end{document}